\documentclass[a4paper,11pt,reqno]{amsart}
\usepackage{mathpazo}
\usepackage{textcomp}
\usepackage{nicefrac} 
\usepackage{xfrac}  
\usepackage{frcursive}
\usepackage[T1]{fontenc}

\usepackage{graphicx}

\usepackage[hidelinks]{hyperref}
\usepackage{amsmath}
\usepackage{amsfonts}
\usepackage{amssymb}
\usepackage{makerobust}
\usepackage{enumerate}
\usepackage{mathrsfs}

\usepackage[all]{xy}

\usepackage{comment}

\MakeRobustCommand\overrightarrow
 
\usepackage{mathtools}

\usepackage[textsize=tiny]{todonotes}

\theoremstyle{plain}
\newtheorem{lemma}{Lemma}[section]
\newtheorem{theorem}[lemma]{Theorem}
\newtheorem{proposition}[lemma]{Proposition}

\newtheorem{corollary}[lemma]{Corollary}

\newtheorem*{convention*}{Convention}

\theoremstyle{definition}
\newtheorem{definition}[lemma]{Definition}

\newtheorem{example}[lemma]{Example}
\newtheorem{remark}[lemma]{Remark}
\newtheorem*{definition*}{Definition}

\theoremstyle{remark}

\newcommand{\smallO}[1]{\ensuremath{\mathop{}\mathopen{}o\mathopen{}\left(#1\right)}}

\newcommand{\groupoid}{\mathcal{G} \rightrightarrows M}

\newcommand{\MagnusTwo}[2]{\mathcal{M}_{{#1}}({#2})}

\newcommand{\MagnusOne}[1]{\mathcal{M}({#1})}

\newcommand{\ad}{\text{ad}}

\newcommand\s[2]{#1\left[[#2]\right]}

\usepackage{enumitem}
\usepackage{enumerate}

\title{Symplectic groupoids for Poisson integrators}

\author{Oscar Cosserat}

\address{LaSIE, UMR 7356 du CNRS \& La Rochelle Université,
Av. Michel Cr\'epeau, 17042 La Rochelle Cedex 1, France}

\email{oscar.cosserat@univ-lr.fr}

\date{\today}

\begin{document}

\begin{abstract}
    We use local symplectic Lie groupoids to construct Poisson integrators for generic Poisson structures. More precisely, recursively obtained solutions of a Hamilton-Jacobi-like equation are interpreted as Lagrangian bisections in a neighborhood of the unit manifold, that, in turn, give Poisson integrators. We also insist on the role of the Magnus formula, in the context of Poisson geometry, for the backward analysis of such integrators.
\end{abstract}

\maketitle

\tableofcontents

\newpage

\section*{Introduction}

In this paper, we address the question of Poisson integrators. Those are examples of so-called geometric integrators, which are numerical methods for solving differential equations, designed to preserve some geometric structure naturally associated to the studied system. The first and the best studied geometric numerical methods are symplectic integrators, also called symplectic schemes\footnote{Throughout this paper we will use ``integrators'', ``schemes'' and ``numerical methods'' as synonyms.}. Symplectic integrators (see e.g. \cite{yoshida}) are suited to discretize the flow of Hamiltonian equations, and as their name suggests, are designed to preserve the symplectic structure in the process. Qualitatively, this results in a better control on the conservation of the energy of the system (\cite{aziz}), even for simulations on large time intervals. Designed in the early eighties, they are now widely used in various applications, like conservative large scale molecular dynamics \cite{Verlet1967}. A lot of works followed, that attempted to preserve various structures naturally associated to the phase space of the system or to the system itself; we are not trying to make a full literature review on the matter, and we orient a motivated reader to \cite{Hairer2006} and references therein for integrators preserving several structures from classical differential geometry, and to a more recent review \cite{LHS} for structures coming from more contemporary ``higher'' and ``generalized'' geometry. 

Poisson geometry permits to generalize simultaneously Hamiltonian mechanics on a symplectic manifolds and Lie group dynamics. Furthermore, it is an efficient tool to study symmetries of a large class of dynamical systems, arising from conservative equations such as the ones of celestial mechanics \cite{Arnol'd1974}, rigid body \cite{Marle1987}, Toda lattices \cite{KhaoulaBenAbdeljellil}, Korteweg-de-Vries equation \cite{Khesin1997}, Lotka-Volterra systems \cite{vanhaecke2016}, to cite a few. Except for the first one, those are associated to non-symplectic Poisson structures.  
A natural question is then the design of numerical methods that take into account this geometry in order to find reliable approximations of solutions. And this question was indeed addressed right after the appearance of symplectic integrators. The first of them were based on an important result that Poisson manifold is foliated into symplectic leaves \cite{weinstein}, the idea being essentially that the dynamics shoud be restricted to a leaf, so that a usual symplectic integrator can be used. The main issue of this approach is that having a Poisson structure where one can explicitly (and globally) describe the leaves is a very strong assumption, so the class of systems where the construction applies is rather small. The next class of papers (\cite{Ge1990}, \cite{Mac1993}) made a step forward in this direction, enforcing the condition of preservation of the leaves of the Poisson foliations, being often not explicit but conceptually more appropriate. However, we have observed (see Example \ref{ex:counter-ex}) that some of these constructions applied naively do not produce the desired results in terms of energy conservation. More recently, the authors \cite{DeLeon2017} have constructed Poisson integrators for dual of Lie algebroids (i.e. fiberwise linear Poisson structures on a vector bundle), understood them through Hamilton-Jacobi equations and Lagrangian bisections (see also \cite{Ge1990}). These are elements that appear in the present study as well, but now for generic Poisson structures.

Indeed, in this paper, we revisit and explain the above mentioned problems in a more conceptual and general framework. We introduce a (stronger) notion of a Hamiltonian Poisson integrator, which takes into account simultaneously the geometry of the phase space (Poisson structure) and the physics of the system (Hamiltonian function). Moreover we make this idea constructive by using the local symplectic groupoid associated to Poisson manifolds. Since the symplectic groupoid inducing this Poisson structure on its unit can be thought of as a bigger foliated space where the foliation has been desingularized, the discritized dynamics we suggest uses heavily the idea to lift the picture to this groupoid and project back at each time step -- with an explicit construction.

The article is organized as follows.
In sections 1 and 2, we introduce the necessary mathematical background for the construction of Hamiltonian Poisson integrators. First,  we adapt the Magnus formula to time-dependent Hamiltonian systems. Second, we explain the concept of families of Lagrangian bisections of symplectic groupoids. This is already enough to formulate the notion of Hamiltonian Poisson integrators and give several properties, like, e.g. backward analysis. Then in section 3, we use an adaptation of the Hamilton-Jacobi equation to make this idea constructive, namely to produce smooth families of Lagrangian bisections inducing Poisson integrators that approximate at any given order the Hamiltonian flow.

In the sequel, $(M,\pi)$ is a Poisson manifold, whose Poisson bracket will be denoted by $\{F,G\}$ for all $F,G \in  \mathcal{C}^{\tiny \infty}(M)$. Also, $\epsilon \in I \subset \mathbb R $ is a real number (thought of as being small and positive when having numerical applications in mind), called \emph{discretization parameter}.

Below is a list of references for several notions that we will not recall:
\begin{enumerate}
    \item Poisson manifolds $\left(M,\pi= \{\cdot, \cdot\}\right)$, \cite{LPV2012,Crainic2021}. The Poisson structure will be denoted by $\pi$ (when considered as a section of $ \wedge^2 TM$) or by $(F,G) \mapsto \{F,G\} $ when considered as biderivation of smooth functions.
    \item Lie groupoids and local Lie groupoids \cite{Mackenzie1987,Weinstein1987}, denoted respectively as $\groupoid$ and $\mathfrak U(M) \rightrightarrows M$. For all considered local or global groupoids, the source shall be denoted by $\alpha $ and the target by $ \beta$.
\end{enumerate}

\section{Hamiltonian Magnus formula}\label{sec:Magnus}

For $A(t)$ a time-dependent linear operator, the Magnus formula  allows to make the time $\epsilon$ flow of a time-dependent linear differential equation $\dot{x}=A(t)x$ of order $1$ as an exponential $x(\epsilon) = {\mathrm{exp}}(B_\epsilon) x(0)$. In general, there are convergence issues that forbid $B_\epsilon$ to be defined out of $A(t) $ for a given value of $\epsilon $, but it is well-defined as a formal series in $\epsilon $.
More generally, the Magnus formula allows to express, up to convergence issues, the flow at a given time $\epsilon$ of a time-dependant left-invariant vector field on a Lie group by an exponential trajectory at time $1$ of a left invariant vector field depending on $\epsilon$ (but not depending on the time $t$). A review on Magnus expansion can be found in \cite{Blanes2008}. The aim of the present section is to adapt the idea to time-dependent Hamiltonian differential equations on a Poisson manifold. 

A \emph{time-dependent function} on a manifold $M $ is a family $(H_t)_{t \in I}$ of functions on $M$ that depend smoothly on the parameter $t$ in the sense that $ (m,t) \mapsto H_t(m)\in \mathcal{C}^{\infty} (M \times I)$. 
For $(M,\pi) $ a Poisson manifold, a time-dependent function $(H_t)_{t \in I} \in \mathcal{C}^\infty(M \times I)$ will be referred to as a \emph{time-dependent Hamiltonian function}. It induces a time-dependent vector field $ X_{H_t} := \{ H_t,\cdot\}$ called \emph{time-dependent Hamiltonian vector field}. 

We call \emph{formal Taylor expansion} of $(H_t)_{t  \in I} $ the formal series 
$$H\left[[\epsilon]\right] := \sum_{i \geq 0} \frac{\epsilon^i}{i!} \left. \frac{\partial^i H_t}{\partial t^i} \right|_{t=0} \, \, \, \in \mathcal{C}^\infty(M)\left[[\epsilon]\right].$$

One must not confuse the formal Taylor expansion of a Hamiltonian function  $(H_t)_{t  \in I} $ (which does not depend on the Poisson structure $\pi $) with a second and more subtle formal series in $\mathcal{C}^\infty(M)\left[[\epsilon]\right]$ defined as follows:

\begin{definition}
The Magnus formal series 
$$\MagnusTwo{\epsilon}{H} 
= 
\sum\limits_{i=0}^{\infty}  \frac{\epsilon^i}{i!} \MagnusOne{H}_i \in \mathcal{C}^{\infty} (M)\left[[\epsilon]\right]$$
of $(H_t)_{t \in I}$ is defined by the formal differential equation:
\begin{equation}\label{magnus}
\left\{
    \begin{array}{ll}
        \MagnusTwo{0}{H}  &= 0\\
        \partial_\epsilon \MagnusTwo{\epsilon}{H}  &= \sum\limits_{i=0}^{\infty} \frac{B_i}{i!} \text{ad}^i_{\MagnusTwo{\epsilon}{H}} \, \s{H}{\epsilon}
    \end{array}
\right.
\end{equation}
where $\text{ad}_{\MagnusOne{H}} = \{\, \MagnusOne{H} , \cdot \, \}$,
$\text{ad}_{\MagnusOne{H}}^i$ is the $i$-th power of the endomorphism $\text{ad}_{\MagnusOne{H}}$, and

$\text{ad}^0_{\MagnusOne{H}} = Id$. 
Here, $(B_i)_{i \in \mathbb N}$ is the Bernoulli sequence, defined by its generating function: $ \frac{x}{\exp(x)-1} = \sum\limits_{i=0}^{\infty}\frac{B_i}{i!}x^i.$

\end{definition}

The terms of the Magnus formal series $\MagnusOne{H}\left[[\epsilon]\right]$ can be computed recursively out of Equation \eqref{magnus}, which ensures its existence and uniqueness.

\begin{remark} There is another expression of the Magnus formal series obtained out of sucessive integration of \eqref{magnus}, which results in the practical formula \cite{Blanes2008}:
\begin{equation}
    \label{Magnus_brackets}
\begin{array}{rcl}
    \MagnusTwo{\epsilon}{H}  &=&\int_0^\epsilon H_t \text{dt} \\ && 
    - \frac{1}{2} \int_0^{\epsilon} \Big\{ \int_0^{t_1} H_{t_2} \, \text{dt}_2, H_{t_1}  \Big\} \,  \text{dt}_1 \\ &&
    + \frac{1}{6} \int_0^{\epsilon} \Big\{\int_0^{t_1} \Big\{ \int_0^{t_2} H_{t_3} \, \text{dt}_3, H_{t_2}   \Big\} \, \text{dt}_2 ,  H_{t_1}  \Big\}\, \text{dt}_1 \\ &&
    + \ldots
\end{array}
\end{equation}
Let us explain the meaning of this expression. Assume we wish to compute the third term $\frac{\epsilon^3}{3!}\MagnusOne{H}_3 $ in the Magnus formal series. For that purpose, it suffices to find the term in $\epsilon^3 $ in each one of the first two terms of \eqref{Magnus_brackets}:
\begin{equation}
\begin{array}{rcl}
    \int_0^\epsilon H_t \text{dt} &=& \epsilon H_0 + \frac{\epsilon^2}{2}   \frac{\partial H_t}{\partial t}_{|t=0}  + \frac{\epsilon^3}{6}   \frac{\partial^2 H_t}{\partial t^2}_{|t=0}  + \cdots  \\   \int_0^{\epsilon} \Big\{ \int_0^{t_1} H_{t_2} \, \text{dt}_2, H_{t_1} \Big\} \, \text{dt}_1 &=& \frac{\epsilon^3}{6}\{ H_0, \frac{\partial H_t}{\partial t}_{|t=0} \} + \cdots\\
\end{array}
\end{equation}
and to add them up.
\end{remark}

\begin{example}
For a time-independent Hamiltonian $(H_t)_{t \in I}$ with $H_t=H$ for all $t \in I$, the Magnus formal series is $\epsilon H$.
\end{example}

\begin{example}
If $M$ and $N$ are two Poisson manifolds, $\phi \colon M \to N$ is a Poisson map and $(H_t)_{t \in I}$ is a time-dependent Hamiltonian on $N$: $\MagnusTwo{\epsilon}{\phi^*H} = \phi^*\MagnusTwo{\epsilon}{H}.$
\end{example}

For any $k \in \mathbb{N}$, we call \emph{$k$-th Magnus truncation Hamiltonian} and we denote by $$\MagnusTwo{\epsilon}{H}^{(k)} := \sum_{i=0}^{k} \frac{\epsilon^i}{i!} \MagnusOne{H}_i$$
the sum of the $k+1$ first terms of the Magnus formal series $\MagnusTwo{\epsilon}{H}.$
By construction, for all given $\epsilon $, $\MagnusTwo{\epsilon}{H}^{(k)} \in \mathcal C^\infty(M)$ is a smooth Hamiltonian function on $M$.

\begin{theorem}
\label{theo:Magnus}
Let $(H_t)_{t \in I} $ be a time-dependent Hamiltonian on a Poisson manifold $(M,\pi)$.
For any  $k \in \mathbb{N}$:
\begin{enumerate}
\item[(i)] the flow $\Phi_{(H_t)_t}^\epsilon$, at time $\epsilon$, of the time-dependent Hamiltonian $(H_t)_{t\in I} \in \mathcal C^\infty(M \times I)$,
 \item[(ii)] and the flow $\Phi_{\MagnusTwo{\epsilon}{H}^{(k)}}^1$, at time $1$, of the $k$-th Magnus truncation Hamiltonian $\MagnusTwo{\epsilon}{H}^{(k)} \in \mathcal C^\infty(M)$,
\end{enumerate}
coincide up to order $k$ in $\epsilon$.

In other words, for all $f \in \mathcal{C}^\infty(M)$ and $  0 \leq j \leq  k $:
\begin{equation}
  \left.  \frac{\partial^j}{\partial \epsilon^j} \right|_{\epsilon = 0} \left(\Phi_{(H_t)_t}^\epsilon - \Phi_{\MagnusTwo{\epsilon}{H}^{(k)}}^1 \right)^* \, f \,= \,0.
\end{equation}

\end{theorem}

\begin{proof}
The computation is a formal Hamiltonian analog of \cite{Blanes2008}.

Set $f_\epsilon^{(k)} \in \mathcal{C}^\infty(M)\left[[\epsilon]\right]$ the formal Taylor expansion of $\Phi_{\MagnusTwo{\epsilon}{H}^{(k)}}^1 {}^*f$, where pull-backs of smooth maps are defined on formal series in an obvious way. The definition of $\MagnusOne{H}$ implies the following equalities of formal series :

\begin{align*}
    \partial_\epsilon (f_\epsilon^{(k)}) &= \Phi_{\MagnusTwo{\epsilon}{H}^{(k)}}^1 {}^* \sum_{i=0}^\infty \frac{1}{(i+1)!} \ad^i_{X_{{\MagnusTwo{\epsilon}{H}^{(k)}}}} . \partial_\epsilon X_{\MagnusTwo{\epsilon}{H}^{(k)}}(f) \\
                                                              &=  \Phi_{\MagnusTwo{\epsilon}{H}^{(k)}}^1 {}^* \sum_{i=0}^\infty \frac{1}{(i+1)!} \ad^i_{X_{\MagnusTwo{\epsilon}{H}^{(k)}}} . \sum_{j = 0}^{k-1} \frac{B_j}{j!} \text{ad}^j_{X_{\MagnusTwo{\epsilon}{H}^{(k)}}} . X_{H_\epsilon} (f) + \smallO{ \epsilon^{k-1}}\\
                                                              &=  \Phi_{\MagnusTwo{\epsilon}{H}^{(k)}}^1 {}^* X_{H_{\epsilon}}(f)  + \smallO{ \epsilon^{k-1}}
\end{align*}

As $\Phi_{(H_t)_t}^0 = \Phi_{\MagnusTwo{0}{H}^{(k)}}^1 = \text{Id},$ the result follows by differentiation.
\end{proof}

\begin{remark}
Theorem \ref{theo:Magnus} can be restated using functions of particular interest in mechanics, namely local coordinates $x$ on $M$:
\begin{equation}
    \forall \; 0 \leq j \leq k, \; \; \frac{\partial^j}{\partial \epsilon^j} \big|_{\epsilon = 0} (\Phi_{(H_t)_t}^\epsilon - \Phi_{\MagnusTwo{\epsilon}{H}^{(k)}}^1 )(x)= 0
\end{equation}
to get an equality of $k$-th jet of curves at $x.$
\end{remark}

\begin{remark}
A particular case of the Magnus formula in the symplectic setting appears in \cite{Koseleff93}, Section 19, where the author studies symplectic integrators for the harmonic oscillator. Up to different conventions, Equation (19.9) is the Magnus formula of the Hamiltonian of Equation (19.11).
\end{remark}

Several time-dependent Hamiltonian vector fields we dealt with in this section arise while studying geometric integrators of Hamiltonian systems that do not depend on time.
Indeed, under some general assumptions, each iteration of a Poisson integrator for a Hamiltonian $H$ is the time $\epsilon $-flow of a time-dependent Hamiltonian $(h_t)$, as will be detailed in \ref{sec:back_analysis}.
To have an integrator at order $k$, we will require the Magnus series $\MagnusTwo{\epsilon}{h}$ of $(h_t)$ to coincide with $\epsilon H$ at order $k$.

\section{Poisson integrators}\label{sec:Poisson_integrators}

In order to define and study Poisson integrators, we recall simple facts of symplectic geometry.

\subsection{Smooth families of Lagrangian bisections}\label{sec:smooth Lag}

\begin{definition}
\label{def:smoothfamily}
Let $V$ be a manifold. A family $ (L_t)_{t\in I}$ of submanifolds of $V$ parametrized by $I$ is said to be a \textit{smooth family of submanifolds} of $V$ if $L_I = \{(x,t) \in V \times I, x \in L_t\}$ is a submanifold of $V\times I$ such that the restriction to $L_I$ of the projection $V \times I \rightarrow I$ is a surjective submersion.
\end{definition}

From now on, we fix $(L_t)_{t\in I}$ a smooth family of submanifolds of $V$, and $ L_I \subset V \times I$ as in Definition \ref{def:smoothfamily}. 
Let $NL_t =TV|_{L_{t}}/ TL_{t}  $ be the normal bundle of $L_t$. We claim that there is a canonically defined smooth section 
$$\left[ \tfrac{\partial L_t}{\partial t}\right] \in \Gamma(NL_{t})$$
called the \emph{normal variation of $(L_t)_{t\in I}$ at $t_0$}.
We start by a definition.

\begin{definition}
A smooth path $ \gamma  \colon J \to V $, for $J \subset I$ an open interval, is said to be an  $(L_t)_{t\in I}$-path if  $ \gamma(s) \in L_s $ for all $s \in J$. Equivalently, an  $(L_t)_{t\in I}$-path is a smooth path $\gamma \colon J \to V $ such that $s \mapsto (\gamma(s),s) $ is valued in the submanifold $L_I \subset V \times I$.
\end{definition}

\noindent
The existence, uniqueness and smoothness of the normal variation follow from the three items of Lemma \ref{lem:normalvariation} respectively.

\begin{lemma}
\label{lem:normalvariation}
Let $ (L_t)_{t \in I}$ be as above.
\begin{enumerate}
    \item Let $t_0 \in I$ and $x \in L_{t_0}$. There exists at least one $(L_t)_{t\in I}$-path $\gamma $, defined in an open neighborhood of $t_0$, such that $\gamma(t_0)=x $. 
    \item For any two paths $\gamma_1, \gamma_2 $ as in the first item,
    $ \dot{\gamma}_1(t_0)-  \dot{\gamma}_2(t_0) \in T_x L_{t_0}$.
    \item  The map assigning to $x \in L_{t_0}$
    the class in the normal bundle
    of the derivative $\dot{\gamma}(t_0) $ of a path as in the first item
    is a smooth section of the normal bundle.
\end{enumerate}
\begin{proof}
The $I$-valued path $t \mapsto t $ lifts
through  $\xymatrix{L_ I \ar@{->>}[r]& I  }$ to a path $\gamma $ with $\gamma(t_0)=(x,t_0) $, because the latter map is a surjective submersion by assumption. 
The two remaining items are straightforward and left to the reader.
\end{proof}
\end{lemma}

\begin{figure}[!h]
    \centering
    \includegraphics[scale = 0.25]{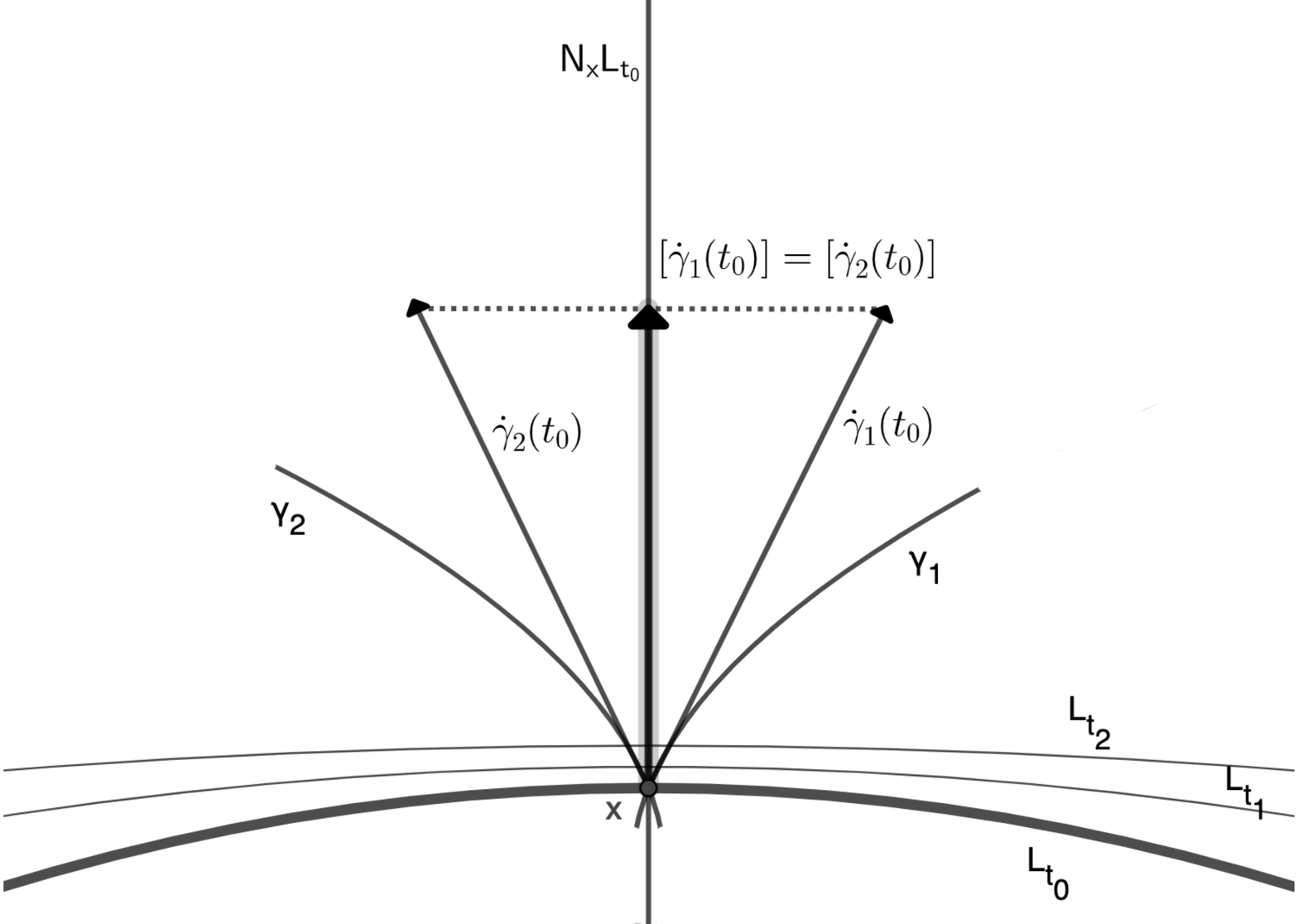}
    \caption{Normal variation of $L_{t_0}$ at $x$ out of two L-paths $\gamma_1$ and $\gamma_2$.}
\end{figure}

From now on, let us assume that $V$ is equipped with a symplectic $2$-form $\omega_V $.
A smooth family of submanifolds $(L_t)_{t \in I} $ is said to be \emph{a smooth family of Lagrangian submanifolds } when all the submanifolds $L_t$ are Lagrangian. Under these assumptions, for all $t \in I$, the normal bundle $TV/TL_t$ is canonically isomorphic to $T^*L_t$ and the normal variation is a family of $1$-forms $\xi_t \in \Omega^1(L_t)$, called \textit{variation form of $(L_t)_{t \in I} $ at $t$}. In equation:
\begin{equation}\label{eq:glue_1forms}
    \omega_V \left( \left. \left[ \tfrac{\partial L_t}{\partial t}(x)\right] \right|_{t} , u \right)  =  \xi_t  (u)   \hbox{ for all $u \in T_x L_t$}
\end{equation}

\noindent
The following Lemma is left to the reader as well:

\begin{lemma}\label{lem_closed}
Let $(L_t)_{t \in I} $ be a smooth family of Lagrangian submanifolds.
The variation form $\xi_t \in \Omega^1(L_t) $ of $(L_t)_{t \in I} $ at $t \in I$ is a closed $1$-form.
\end{lemma}

\begin{definition}
We call \emph{exact} a smooth family of Lagrangian bisections $(L_t)_{t \in I} $ such that its corresponding variation $1$-forms $(\xi_t)_{t \in I} $ are exact; we call  \emph{variation functions} their primitives, i.e.  some time-dependent functions $(h_t)_{t \in I} $  such that $ dh_t = \xi_t$ for all $t \in I$.
\end{definition}

\begin{remark}
For a given smooth family of \emph{exact} Lagrangian bisections,  the family of time-dependent functions $(h_t)_{t \in I} $ is unique up to a function that depends only on $t$. 
\end{remark}

Here are two important classes of smooth families of Lagrangian submanifolds.

\begin{example}\label{Ham_flow}
Let $V$ be a symplectic manifold, and $H \in C^\infty(V) $ a Hamiltonian function whose Hamiltonian vector field admits a flow for all $t \in \mathbb R$. For every Lagrangian submanifold $L \subset V,$ the family $L_t = \phi^t_H (L) $ is an exact smooth family of Lagrangian submanifolds. An $L$-path starting at $x \in L$ is given by the flow of $H$ $(\phi^t_H(x))_t$ and the variation form of $L$ at $t$ is the restriction to $L_t$ of the exact form $dH$.
\end{example}

\begin{example}\label{1-forms}
Let $T^*Q$ be a cotangent bundle. For every smooth family of closed one-forms $(\zeta_t)_{t\in I}$ on $Q,$ their graphs $L_t = \bar\zeta_t = \{\zeta_t(x), x \in Q \}$ are a smooth family of Lagrangian submanifolds. An $L$-path starting at $x \in \bar \zeta_0$ is given by the 1-forms $(\zeta_t(x))_t$ and the variation form at $t$ is $\tau_{|\bar\zeta_t}^* \partial_t\zeta_t,$ where $\tau$ is the cotangent projection and $\tau_{|\bar\zeta_t}$ its restriction to $\bar\zeta_t.$
\end{example}

Variation forms behave well with respect to symplectomorphisms, as explained in the following example.

\begin{example}\label{Symplecto}
Let $(V,\omega_V)$ and $(W,\omega_W)$ be two symplectic manifolds,\\$\phi \colon V \stackrel{\sim}{\longrightarrow} W$ a bijective symplectomorphism and $(L_t)_t$ a smooth family of Lagrangian submanifolds with variation 1-forms $\xi_t.$ Then, $\tilde{L}_t = \phi(L_t)$ is also a smooth family of Lagrangian submanifolds with variation 1-forms $\tilde{\xi_t}$ that verify $\phi^*\tilde{\xi_t} =\xi_t.$
\end{example}

\subsection{Usual Hamilton-Jacobi equation revisited}\label{sec:classic_Ham_Jac}
We use variation forms to reinterpret the usual Hamilton-Jacobi equation in terms of smooth families of Lagrangian submanifolds. Consider a Hamiltonian $H \in \mathcal C^\infty(T^*Q)$ on the cotangent bundle $T^*Q$ of a manifold $Q$. The Hamilton-Jacobi equation consists in looking for a family of functions $S_t \in C^\infty(Q \times Q)$, depending smoothly in $t$ in some interval $J \subset \mathbb R$, such that for every $(q,\bar q) \in Q\times Q$ and every $t \in J$:
\begin{equation}\label{HJ_usual}
    \partial_t S_t (q, \overline{q}) = H( \text{d}_q S_t(q, \overline{q}) )
\end{equation}
where
$\text{d}_q S_t(q, \overline{q}) \in T_q^* Q$ is the differential of $S(\cdot, \bar{q})$ at the point $q$, so that
\\
$\nolinebreak{(q, \bar q) \mapsto H( \text{d}_q S_t(q, \overline{q}) ) \in C^\infty(Q \times Q)}$.
 
\begin{theorem}[Hamilton-Jacobi theorem for a cotangent bundle]\label{thm:HamiltonJacobi}
If $(S_t)_{t \in J}$ verifies \eqref{HJ_usual}, if there exist open subsets $\mathcal U \subset  Q \times Q $ and $\mathcal V \subset T^*Q $ between which $(q,\overline{q}) \mapsto  \text{d}_{q} S_t(q, \overline{q} )$ is a diffeomorphism for every $t \in  J$ and if their exists $\epsilon_0 \in J$ such that the Hamiltonian flow $\Phi_H^{\epsilon_0}$ of $H \in \mathcal C^\infty(T^*Q)$ at time $\epsilon_0$ is given by: $$ \Phi_H^{\epsilon_0}(\zeta) = \text{d}_{\overline{q}} S_{\epsilon_0} (q, \overline{q}) \hspace{1cm} \hbox{ $\forall \zeta \in \mathcal V \subset T^* Q $} $$
where $(q,\overline{q}) \in Q \times Q$ is the unique element in $\mathcal U$ that satisfies
$ \zeta = -\text{d}_{q} S_{\epsilon_0}(q, \overline{q}),$\\
Then the Hamiltonian flow $\Phi_H^\epsilon$ of $H \in \mathcal C^\infty(T^*Q)$ at time $\epsilon \in J$ is given by: $$ \Phi_H^\epsilon(\zeta) = \text{d}_{\overline{q}} S_\epsilon (q, \overline{q}) \hspace{1cm} \hbox{ $\forall \zeta \in \mathcal V \subset T^* Q $}.$$
\end{theorem}

Quite often \cite{Marsden1998}, elements in $T^*Q$ are defined as pairs $(q,p)$ with $q \in Q$ and $p \in T^*_q Q $, and the map $\Phi_\epsilon^H\colon (q,p) \mapsto (\bar q,\bar p)$ above is then seen as being implicitly defined by: 
\begin{equation}
    \left\{
    \begin{array}{ll}
        \overline{q} &= \partial_q S_\epsilon(q,\overline{q})\\
        \overline{p} &= - \partial_{\overline{q}} S_\epsilon(q,\overline{q})
    \end{array}
    \right.
\end{equation}

Let us use the tools developped in subsection \ref{sec:smooth Lag} to give an interpretation and a proof of Theorem \ref{thm:HamiltonJacobi}.

Geometrically, two families of Lagrangian submanifolds are involved here:
\begin{enumerate}
    \item Since $\Phi_t^H\colon T^* Q \to T^*Q $ is a symplectomorphism, its graph $\underline{\Phi_H^t} = \{ (x, \phi^t_H(x)) \in T^*Q \times T^*Q, x \in T^*Q\}$, is a Lagrangian submanifold of $T^*Q \times T^*Q$ equipped with the product symplectic form with appropriate signs. As in \ref{Ham_flow}, variation forms of $(\underline{\Phi_H^t})_t$ are $(\xi_t)_t = (\Phi_H^{-t} {}^*\text{d}H)_t,$ where $\Phi_H^{-t}$ is understood as a map $\Phi_H^{-t}:\underline{\Phi_H^t} \to T^*Q.$
    \item For every $t \in J$, the graph of the exact form $\underline{\text{d}S_t}$ is a Lagrangian submanifold of $T^*(Q\times Q)$ equipped with its canonical symplectic form. As in \ref{1-forms}, variation forms of $(\underline{\text{d}S_t})_t$ are $ (\tilde{\xi_t})_t = (\tau_{|\text{d}S_t}^*\text{d} \frac{\partial}{\partial_t} S_t)_t,$ for $\tau$ the cotangent projection.
\end{enumerate}

Now, the two symplectic manifolds above are canonically symplectomorphic: 

\begin{equation}\label{isomorphism2}
    \begin{array}{cccc}
        \Psi \colon &  T^*Q \times T^*Q                   & \to     & T^*(Q \times Q) \\
                    &  (\zeta(q),  \bar \zeta (\bar q)) & \mapsto & \zeta(q) - \bar \zeta(\bar q)
    \end{array}.
\end{equation}
As in \ref{Symplecto}, both variation forms are related by $\Psi,$ hence:

\begin{equation}\label{eq:diff_class_Ham_Jac}
    \text{d}\partial_t S_t (q, \bar q) = \text{d} (\text{d}_q S_t)^*H.
\end{equation}
It is clear that a time-dependent constant can be added to generating functions $(S_t)_t$. \eqref{eq:diff_class_Ham_Jac} is Hamilton-Jacobi equation \eqref{HJ_usual} up to a time-dependent coboundary, that completes the proof of theorem \ref{thm:HamiltonJacobi}.

\indent As a conclusion, one geometric interpretation of Hamilton-Jacobi equation is that the canonical symplectomorphism \eqref{isomorphism2} above intertwines the two families of Lagrangian submanifolds $(\underline{\text{d}S_t})_t$ and $(\underline{\Phi_H^t})_t,$ and the resulting equation is the one of their variation forms.

\begin{remark}
 The solution $(S_t)_t$ of \eqref{HJ_usual} may be singular at $t=0,$ since the diagonal $\Delta = \{x,x\}_{x \in T^*Q}$ of $T^*Q \times T^*Q$ is sent by $\Psi$ to $\{x,-x\}_{x \in T^*Q}$ which is not the graph of a globally defined differential form on $Q \times Q.$
For instance, when $$H: T^*\mathbb{R}^d \to \mathbb{R}: (q,p) \mapsto V(q) + K(p)$$ is a separable fiberwise convex Hamiltonian and $f$ is the Legendre transform of $K,$ $\nabla f = (\nabla K)^{-1}$ and the symplectic Euler scheme $(q,p) \mapsto (\tilde{q},\tilde{p})$ can be rewritten as:
    \[
    \left\{
    \begin{array}{ccc}
        p =& \nabla f(\frac{q - \tilde{q}}{\epsilon}) + \epsilon \nabla V(q) &= \partial_q S_\epsilon(q,\tilde{q})\\
        \tilde{p} =& \nabla f(\frac{q-\tilde{q}}{\epsilon}) &= \partial_{\tilde{q}} S_\epsilon(q,\tilde{q})
    \end{array}
    \right.,
    \]
for $S_\epsilon(q,\tilde{q}) = \epsilon V(q) - \epsilon f(\frac{q - \tilde{q}}{\epsilon}).$ Since $K$ is convex, for any $(q,\bar q)$ outside the diagonal, $\lim_{\epsilon \to 0} S_\epsilon(q,\tilde{q}) = \infty.$\\
This have consequences because in numerical schemes, the step is a small number. For instance, in \cite{DeLeon2017}, where a formalism for fiberwise linear Poisson structures is developped, the authors get rid of the singularity by a local non-canonical change of coordinates.

Using an embedding of the local sympelctic groupoid in the cotangent bundle of the unit manifold, we will see in section \ref{sec:Ham_Jac_new} a different kind of Hamilton-Jacobi equation for which $S_0=0$.
\end{remark}

\begin{remark}
    In theorem \ref{thm:HamiltonJacobi}, one might replace $Q \times Q$ by a groupoid $G \rightrightarrows Q,$ $T^*Q$ by the dual $A^*$ of the Lie algebroid of $G$ and $T^*Q \times T^*Q$ by the cotangent groupoid $T^*G \rightrightarrows A^*$ of $G.$ This is based on the classical observation that $T^*G$ is a symplectic groupoid integrating the Poisson manifold $A^*.$  Then, one would obtain an analog of Theorem 7 in \cite{DeLeon2017}.
\end{remark}

\subsection{Symplectic groupoids}\label{sec:symp_grpid}

The Lagrangian submanifolds we are interested in lie in a neighborhood of the Poisson manifold in its symplectic groupoid, a framework that we introduce now, following \cite{Weinstein1987} (see \cite{Crainic2021} for a review on the subject).

\begin{definition}
Let $\groupoid$ be a Lie groupoid over $M.$ A symplectic 2-form $\Omega \in \Omega^2(\mathcal G)$  is said to be multiplicative if the graph of the product
$$\Big\{ (g_1, g_2, g_1 . g_2) \in \mathcal{G}^3, \alpha(g_2) = \beta(g_1) \Big\}$$
is Lagrangian for the symplectic form $\text{pr}_1^*\Omega + \text{pr}_2^*\Omega - \text{pr}_3^*\Omega$, where $\text{pr}_i : \mathcal G^3 \to \mathcal G$ is the projection over the $i$-th factor. The pair $(\groupoid,\Omega)$ is then called a symplectic groupoid.
\end{definition}

Let us recall some useful properties of symplectic groupoids.
\begin{enumerate}
    \item The unit manifold $M$ is a Lagrangian submanifold in $(\Gamma,\Omega)$ and comes equipped with a natural Poisson structure $\pi$ such that the source $\alpha\colon \Gamma 
    \to M$ is a Poisson map. Also the target $\beta\colon \Gamma\to M$ is an anti-Poisson map. 
    \item The Lie algebroid of $\Gamma$ is isomorphic to $T^*M$: its anchor is $\pi^\sharp
    \colon T^*M\to TM$. Its leaves are the symplectic leaves of $\pi$. Also, since a $1$-form  $\nu\in\Omega^1(M)$ is a section of the Lie algebroid, it defines a right-invariant vector field and a left-invariant vector field on $\groupoid$, respectively denoted by $\overrightarrow{\nu}$ and $\overleftarrow{\nu}$ and associated, under the isomorphism $\Omega^{\flat}\colon T\mathcal G\simeq T^* \mathcal G$, to the left and right invariant $1$-forms $\alpha^*\nu$ and $\beta^*\nu$.
    \end{enumerate}
Although not every Poisson manifold $(M,\pi)$ is the unit manifold of a symplectic groupoid, every Poisson manifold is the unit manifold of a local symplectic groupoid, see e.g. \cite{Crainic2011}, \cite{Zung2005}, said to \emph{integrate} $(M,\pi)$. Two local symplectic groupoids integrating the same Poisson manifold are isomorphic in a neighborhood of $M$.

The relation between the local symplectic groupoid of a Poisson manifold and  Poisson integrators comes from the following theorem about \emph{bisections}, i.e. submanifolds $L\subset \mathcal G$ to which the restrictions of both source and target maps are diffeomorphisms onto $M$. Notice that any bisection $L
\subset \mathcal G$ induces a diffeomorphism $\phi_L:=\beta_{|_L}
\circ \alpha_{|_L}^{-1}$ of the unit manifold $M$.

\begin{proposition}
 [\cite{Weinstein1987}]
Let $(M,\pi)$ be a Poisson manifold and $(\groupoid, \Omega)$ a local symplectic groupoid integrating it. 
If a bisection $L\subset \mathcal G$ is Lagrangian, then:
\begin{enumerate}
    \item the induced diffeomorphism $\phi_L\colon M\to M$ is a Poisson automorphism,
    \item provided that the fibers of the source map are connected, for all $x \in M,$ $\phi_L(x)$ and $x$ belong to the same symplectic leaf.
\end{enumerate}
\end{proposition}

\noindent
We are now interested in smooth families of Lagrangian
submanifolds $(L_t)_{t \in I} $ of a symplectic groupoid $\mathcal G $, where $I \subset \mathbb R $ is an interval containing $0$, that happen to be bisections for all $t \in I$.
From now on, such an $ (L_t)_{t \in I}$ shall be refered to as a \emph{smooth family of Lagrangian bisections}.

\begin{example}[Lagrangian bisections of the symplectic groupoid of a symplectic manifold]\label{ex:symp_manifold}
Any smooth family of symplectomorphism $(\phi_t)_t$ of a symplectic manifold $(M, \omega)$ is the flow of a time-dependent vector field related by $\omega$ to a time-dependent closed form $(\xi_t).$ Consequently, any smooth family of Lagrangian bisection $(L_\epsilon)_{\epsilon \in I}$ will be, up to a choice of the first and the second factor in $M \times M,$ of the form $\{(x,\phi_\epsilon(x)) , x \in M\}_{\epsilon \in I}.$

For instance, for any solution $S \in \mathcal{C}^{\tiny \infty}(I \times Q \times Q)$ of Hamilton-Jacobi equation as in section \ref{sec:classic_Ham_Jac}, a smooth family of Lagrangian bisections of the pair groupoid is given by $\Psi(\underline{\text{d}S_t})_t$ where
$$\Psi(\underline{\text{d}S_t}) = \{ \left(\text{d}_q S_t(q, \bar q),
-\text{d}_{\bar q}  S_t(q, \bar q) \right), (q, \bar q) \in Q \times Q \} \subset T^*Q \times T^*Q.$$
\end{example}

\begin{example}[Lagrangian bisections of the symplectic groupoid of the dual of a Lie algebra]
Associate $T^*G$ with $G \times \mathfrak{g}^*$ by left translations. It integrates the canonical Poisson structure on $\mathfrak{g}^*,$ where the source is the projection on $\mathfrak{g}^*.$ Then any smooth family of Lagrangian bisections is of the form $$\forall t \text{, } L_t \stackrel{\text{def}}{=} \{ (\rho_t(p), p) \in T^*G, p \in \mathfrak{g}^*\}$$ where $\rho$ is a smooth family of sections of the source such that 
$$ L_{(\rho_t)^{-1}} {}_* \text{d}_x \rho_t: \mathfrak{g}^* \rightarrow \mathfrak{g}$$
is symmetric for the dual pairing for all $x \in \mathfrak{g}^*.$
The corresponding Poisson automorphism is $ p \in \mathfrak{g}^* \mapsto \text{Ad}^*_{\rho_t(p)}.p $
\end{example}

\begin{remark}\label{rk:Ham_Poisson_integrator}
Any exact family of Lagrangian bisections $L$ induces naturally a Hamiltonian Poisson integrator of timestep $\Delta t$
$$ \begin{array}{ll}
 M &\longrightarrow M \\
 x &\longmapsto \beta \circ (\alpha_{|L_{\Delta t}})^{-1}(x)
\end{array}$$
in the sense of the definition \ref{def:ham_pois_int} below.

This procedure allows to construct many Poisson automorphisms that not only stay in the same symplectic leaf when we iterate them but also are Hamiltonian trajectories. This is a natural property to ask to a Poisson scheme. The reader may notice that given a Hamiltonian $H$ on $M,$ one does not know its flow and the bisections $L_t = \Phi^t_{\overrightarrow{H}}(M)$ are consequently not generically computable. In section \ref{sec:Ham_Jac}, we explain how Hamilton-Jacobi equation on the symplectic groupoid produces Lagrangian bisections such that the induced Poisson integrator approximates a Hamiltonian flow at any desired order in the timestep.
\end{remark}

By Lemma \ref{lem_closed}, the variation form is a closed $1$-form on $L_t $ for all $t \in I$. Using $(\alpha_{|_{L_t}}^{-1})^*\colon \Omega^1 (L_t) \to \Omega^1 (M)  $, the variation $1$-forms of $(L_t)_{t \in I}$ become a smooth family $(\xi_t )_{t \in I} $ of closed $1$-forms in $\Omega^1(M)$, that we still call \emph{the variation $1$-forms of $(L_t)_{t \in I} $}, with a slight abuse of notation. Before stating the proposition that relates $\xi$ and $L$: 

\begin{lemma} \label{lem:11corresp}
Let $(L_t)_{t \in I}$ and $(\xi_t )_{t \in I} $ be as above. 
The time $t$-flow of the time dependent vector field $\overrightarrow{\xi_t }=(\Omega^{\flat})^{-1}(\alpha^* \xi_t)$ restricts to a diffeomorphism from $ L_0$ to $L_{t} $.
\end{lemma}
\begin{proof}
By definition of variation forms, $\Phi_{\overrightarrow{\xi}}^t(L_0) \subset L_t.$\\
In order to prove the other inclusion, let $x \in L_{t}$ and set $x_0 = (\Phi^t_{\overrightarrow{\xi}})^{-1}(x).$ We are left to prove that $x_0 \in L_0.$ Since $\overrightarrow \xi$ is complete on $I$ and \eqref{eq:glue_1forms} ensures that the flow $\Phi_{\overrightarrow \xi}$ is locally an $L$-path:

$$ \forall \; 0 \leq u \leq t, \exists \; \epsilon_0 > 0, \forall \; | \epsilon | \leq \epsilon_0, \Phi_{\overrightarrow{\xi}}^{u-\epsilon}(x_0) \in L_{u-\epsilon} $$
For any $u \in [0,t],$ $(\Phi^u_{\overrightarrow \xi})^{-1}(x) \in L_u,$ hence the result.
\end{proof}

Lemma \ref{lem:11corresp} says that smooth families  $(L_t)_{t \in I} $ of Lagrangian bisections in a symplectic groupoid can be recovered from $L_0 $ and from their associated variation forms $(\xi_t) \in Z^1(M)$. Not every pair $(L_0,(\xi_t))$ gives a family of Lagrangian bisections, because the flow of $\overrightarrow{\xi_t}$ may not be defined for all times $t\in I.$ However, the correspondence works under relatively mild assumptions:

\begin{proposition}
\label{prop:11corresp}
Let $I$ be an interval containing 0. In a symplectic groupoid $\groupoid$, there is a one-to-one correspondence between:
\begin{enumerate}[label=(\roman*)]
    \item smooth families of Lagrangian bisections $(L_t)_{t \in I}$ of $ (\mathcal G, \Omega)$,
    \item pairs made of a Lagrangian bisection $L_0 $ and a smooth family of closed one-forms on the base $(\xi_t)_{t \in I}$ such that the vector field  $\pi^\#(\xi_t)$ is a complete vector field on $M$.
\end{enumerate}
\end{proposition}

\begin{remark}
A smooth family of Lagrangian bisections $(L_t)_{t \in I}$ is exact if and only if there exist global $L$-paths that are left-invariant time-dependant Hamiltonian trajectories :
$$ \exists \; H \in \mathcal{C}^\infty(M \times \mathbb{R}), \forall \; t_0 \in I, \forall \; x \in L_{t_0}, \exists \; \gamma \text{ an L-path, } \left\{
\begin{array}{ll}
    \gamma(t_0) = x\\
    \dot{\gamma} = X_{\alpha^* H}
\end{array}
\right.$$
\end{remark}

\begin{proof}[Proof of Proposition \ref{prop:11corresp}]

From lemma \ref{lem:11corresp}, for any $\epsilon \in I,$
\begin{equation}
   L_\epsilon = \phi^\epsilon_{(\overrightarrow{\xi_t})_t}(L_0).
\end{equation}
Consequently, two smooth families of Lagrangian bisections admitting the same variation forms and corresponding at $0$ are equals.

Now, set $L = (\phi^\epsilon_{(\overrightarrow{\xi_t})_t}(L_0))_\epsilon.$ The smothness and bisection properties are clear. We prove that $L_\epsilon$ is a Lagrangian submanifold. \\
Indeed, the flow of a left-invariant vector field is a symplectomorphism if and only if the corresponding 1-form on the base is closed. To verify this claim, set $\Pi \in \Gamma(\wedge^2(\mathcal{G}))$ the Poisson tensor corresponding to $\Omega.$ For any $f,g \in \mathcal{C}^\infty(M)$:

\begin{align*}
    \mathscr{L}_{\overrightarrow{f.\text{d}g}} \Pi &= \alpha^* f \mathscr{L}_{\overrightarrow{\text{d}g}} \Pi + \overrightarrow{\text{d}f} \wedge \overrightarrow{\text{d}g}\\
                                                     &= \overrightarrow{\text{d}f \wedge \text{d}g}\\
                                                     &= \overrightarrow{\text{d}(f \text{d}g )}.
\end{align*}
Then :
\begin{equation}\label{mult}
    \mathscr{L}_{\overrightarrow{\xi_\epsilon}} \Pi = \overrightarrow{\text{d}\xi_\epsilon} = 0
\end{equation}
and that concludes the proof.
\end{proof}

\begin{remark}
Equation \eqref{mult} comes out from the multiplicativity of $\Pi$ and is a particular case of a much more general correspondence between multiplicative polyvector fields on the groupoid and differentials on its algebroid, cf. theorem 2.34 of \cite{CLG2012}.
\end{remark}

\noindent
Examples will be given in Section \ref{sec:Pois_integ}, except for the following two examples, that connect with symplectic geometry.

\begin{example}
The example \ref{ex:symp_manifold} already relates smooth family of closed 1-forms on a symplectic manifold with smooth family of Lagrangian bisections of the associated pair groupoid.
\end{example}

\begin{example}\label{ex:Weinsteins_embedding}
According to Weinstein's theorem \cite{Weinstein1971}, every Poisson manifold $(M,\pi)$ integrates to a local symplectic groupoid structure $(\groupoid, \omega) $ where $\mathcal G $ is a neighborhood $\mathfrak U(M) $ of the zero section of $T^* M $ and $\omega=\omega_{can} $ is the
restriction to $\mathfrak U(M) $ of the
canonical symplectic $2$-form of $T^* M $. To every smooth family $(L_t) $ of Lagrangian bisections with $L_0=M $ and $L_t \subset \mathfrak U(M) $, we can therefore associate two different kinds of families of closed $1$-forms.
\begin{enumerate}
    \item We can forget the groupoid structure, and say that on $T^*M $, each one of the $L_t $ is the graph of a closed $1$-form:  
 $$ L_t = \{ \zeta_t|_m , m\in M\} $$
for $\xi_t $ a closed 1-form on $M.$
\item Alternatively, one can forget that $\mathcal G$ has been identified to $T^* M$, and use Proposition \ref{prop:11corresp} to associate a family $\xi_t $ of closed $1$-forms. 
\end{enumerate}
Both families of closed $1$-forms are in general different, but related by the equality of the variation forms of their corresponding families of Lagrangian submanifolds. As 1-forms on $L_t$ for all $t$:
$$  \alpha^{*} \xi_t = \tau^* \frac{d \zeta_t}{dt} $$
where $\tau \colon T^*M \to M $ is the natural projection.
\end{example}

\subsection{Poisson integrators and their backward analysis}\label{sec:back_analysis}

Poisson integrators appearing in the literature may be understood as germs of Lagrangian bisections. A particular case of this principle is developped for fiberwise linear Poisson structures on the dual of a Lie algebroid in \cite{DeLeon2017}.

Let us consider a Hamiltonian vector field, i.e. a differential equation of the type
$$ \dot{x} = \pi_{x(t)}^{\#} (d_{x(t)} H) =  X_{H}|_{x(t)}  $$
where $(M,\pi) $ is a Poisson manifold and $H \in \mathcal{C}^\infty(M)$ a Hamiltonian function. Out of \cite{Hairer2006}, a reasonable definition of a \emph{Poisson integrator of order $k\geq 1 $ for $H$} might be defined by the following three conditions :
\begin{enumerate}
    \item $\phi_\epsilon$ agrees with the time-$\epsilon $ flow of $X_H $ up to order $k$ in $\epsilon $,
    \item $\phi_\epsilon$ is a Poisson diffeomorphism for all $\epsilon \in I$,
    \item $\phi_\epsilon$ maps each leaf to itself (through a map which is necessarily a symplectic diffeomorphism).
\end{enumerate}

The purpose of a Poisson integrator is to choose a particular value ${\Delta t}$ of $\epsilon$, called \emph{timestep}, then consider the iterations of the diffeomorphism $\phi_{\Delta t}$. The hope is of course that the $n$-th iterations remain good approximations of the flow of $X_H $ at time $n\Delta t$ for large $n$.

In the particular case of symplectic integrators, the theoretical ground of their good behaviour is their backward analysis. Indeed, any smooth family of symplectomorphisms $(\phi_t)_t$ is the flow of a time-dependent vector field $(X_t)_t$ related through the symplectic form to the flow of a closed 1-form. So any symplectic integrator for $H$ at order $k$ is locally the flow of a time-dependent Hamiltonian $(h_t)_t$ such that $h_0 = H.$ The order $k$ of the method is then related to the order at which the initial Hamiltonian $H$ equals $(h_t)_t$: $H = h_t + \smallO{t^{k-1}}.$

In this context, an important feature of Poisson integrators is that it is not always true anymore. There exists a smooth family of Poisson automorphisms, even staying on the same symplectic leaf, that are not a flow of a time-dependent Hamiltonian, because of so-called outer-automorphisms. They are measured by the first Poisson cohomology group of the Poisson manifold. This makes a huge difference with symplectic schemes, for which this property is automatically verified, at least locally. Here is an example of such a phenomenon:

\begin{example}\label{ex:counter-ex}
For the Poisson tensor $\pi = (x^2 +y^2)\partial_x \wedge \partial_y$ on $\mathbb{R}^2$, a Poisson integrator for $H : (x,y) \to \frac{x^2 + y^2}{2}$ of order $k$ and step $\Delta t$ is

\begin{equation}
    \begin{pmatrix}
        x_{n+1}\\
        y_{n+1}
    \end{pmatrix}
    = e^{(\Delta t)^k}
    \begin{pmatrix}
    \cos{\Delta t} & -\sin{\Delta t}\\
    \sin{\Delta t} & \cos{\Delta t}
    \end{pmatrix}
    .
    \begin{pmatrix}
        x_{n}\\
        y_{n}
    \end{pmatrix}
\end{equation}
and behaves remarkably bad for long simulations: for any norm $\|.\|$ and initial point $(x_0,y_0) \neq 0_{\mathbb{R}^2},$

\begin{equation*}
    \| (x_n,y_n) - \Phi^{n\Delta t}_H(x_0,y_0) \| \underset{n \to +\infty}{\longrightarrow} +\infty.
\end{equation*}
As in the general case, this phenomenon is explained by the fact that the first Poisson cohomology group $H^1_\pi$ is locally non-trivial around $0$ : there exists no neighborhood $U$ of $0$ such that $H^1_\pi(U) = \{0\}.$ Indeed, $H^1_\pi$ is generated by rotations and dilations. In other words, there exist smooth families of Poisson automorphisms $(\phi_t)_t$ such that $\phi_0 = \text{Id}$ but $\phi$ is not a flow of a time-dependent Hamiltonian.  This issue comes from the singularity of $\pi$ at $0$.
\end{example}

The previous example suggests to make stronger assumptions to define a notion of Poisson integrator:

\begin{definition}[Hamiltonian Poisson integrator]\label{def:ham_pois_int}
Let $(M,\pi)$ a Poisson manifold and $H \in \mathcal{C}^\infty(M)$ a Hamiltonian on $M.$ A smooth family of diffeomorphisms of $M$ $(\phi_\epsilon)_\epsilon$ is a \emph{Hamiltonian Poisson integrator of order $k\geq 1 $ for $H$} if:
\begin{enumerate}
    \item $\phi_\epsilon$ is a Poisson diffeomorphism,
    \item there exists $(h_t)_t$ a time-dependent Hamiltonian such that 
    \begin{enumerate}
        \item $h_t = H + \smallO{t^{k-1}}$
        \item $\phi_\epsilon = \Phi^\epsilon_{(h_t)_t}$ is the time-$\epsilon$ flow of $h.$
    \end{enumerate}
\end{enumerate}
\end{definition}
It follows easily that

$$\phi_\epsilon = \Phi^\epsilon_{H} + \smallO{t^{k}}$$
in the sense of \ref{theo:Magnus}.

We can now state the main result of this section, which is the core of the explicit constructions of Poisson integrators that will be presented in the sequel. We recall that given an exact family of Lagrangian bisections $L$ on $(\groupoid,\omega),$ their variation functions $(h_t)_t \in \mathcal C^\infty(M \times I)$ denote the pull-back by the source of closed 1-forms obtained from $L$-paths through $\omega.$

\begin{theorem}\label{theorem_Poisson_integrators}
Let $(M,\pi) $ be a Poisson manifold, $(\groupoid,\omega)$ a local symplectic groupoid integrating it and $k \geq 1.$  
For every smooth family $(L_t)_{t \in I}$ of exact Lagrangian bisections such that $L_0=M$ and with variation functions $(h_t)_{t \in I},$ if the Magnus series $\MagnusTwo{\epsilon}{h}$ of $(h_t)_{t \in I}  \in \mathcal C^\infty(M \times I)$ coincides with $\epsilon H$ at order $k$, the induced family of diffeomorphisms $(\phi_{L_t})_{t\in I}$ is a Hamiltonian Poisson integrator of order $k$ for $H$.
\end{theorem}

\begin{remark}

We invite the reader to understand Theorem \ref{theorem_Poisson_integrators} as meaning that, provided a symplectic groupoid integrating a Poisson structure is entirely known and computable, then finding a Hamiltonian Poisson integrator reduces to a Magnus series question. This is the first part of the construction we have announced in the introduction.
\end{remark}

\begin{proof}[Proof of theorem \ref{theorem_Poisson_integrators}]
Since $\phi$ is induced by $L,$ it is a Hamiltonian Poisson integrator for $H$, of time-dependent Hamiltonian $h.$\\
We compute its order. For any $f \in \mathcal{C}^{\tiny \infty}(M)$:

\begin{align*}
    \phi_\epsilon^{*}f &= (\Phi^\epsilon_{(h_t)_{t \in I}})^*f \\
    &= (\Phi^1_{\epsilon H})^* f + \smallO{t^{k}}\\
    &= (\Phi^\epsilon_{H})^*f + \smallO{t^{k}}.
\end{align*}
\end{proof}

\begin{corollary}
Let $(M,\pi) $ be a Poisson manifold and $(\groupoid,\omega) $ a local symplectic groupoid integrating it.  
For every smooth family $(L_t)_{t \in I}$ of exact Lagrangian bisections with $L_0=M$,
if  $\left.\frac{\partial L_t}{\partial t}\right|_{t=0} = d H$, then its induced diffeomorphisms $(\phi_{L_t})_{t\in I}$ are Hamiltonian Poisson integrators of order $1$ for $H$.
\end{corollary}

\subsection{Examples of Poisson integrators}\label{sec:Pois_integ}

The description of Poisson integrators in Theorem \ref{theorem_Poisson_integrators} unifies already known constructions: the classical Euler-Symplectic scheme \cite{Hairer2006} (see Example \ref{ex:EulerSymplectic}), the mid-point method for the harmonic oscillator (see Example \ref{ex:mid_point}), and the Kahan discretization of Lotka-Volterra system already described by Pol Vanhaecke \cite{vanhaecke2016} (see Example \ref{ex:Kahan}).

\begin{example}[Euler symplectic scheme for a separable Hamiltonian]
\label{ex:EulerSymplectic}
For a general Hamiltonian $H \in \mathcal{C}^\infty(T^*\mathbb{R}^d)$ and $(q,p)$ cotangent coordinates, the symplectic Euler scheme (\cite{Hairer2006}) is:

\begin{equation}
\label{eq:PoissonBasic}
    \left\{
    \begin{array}{ll}
        q_{n+1} = q_n + \Delta t \frac{\partial H}{\partial p}(q_n,p_{n+1})\\
        p_{n+1} = p_n - \Delta t \frac{\partial H}{\partial q}(q_{n},p_{n+1})
    \end{array}
\right. .
\end{equation}
Let us interpret this implicit (in the generic case) numerical scheme as a Poisson integrator at order $1$ for $H$, in the sense of Theorem \ref{theorem_Poisson_integrators}.  
We use the notations of the latter theorem:
\begin{enumerate}
    \item[$(M,\omega)$] On $M= T^*\mathbb{R}^d $, we denote by $(q,p)$ some canonical cotangent coordinates, and we consider the canonical Poisson structure associated to the symplectic $2$-form $\omega = \sum_{i=1}^d dp_i \wedge d q_i $.
    \item[$(\mathcal G,\Omega)$] the symplectic groupoid integrating the Poisson manifold $T^*\mathbb{R}^d$ is the pair groupoid $\mathcal G :=T^*\mathbb{R}^d \times T^*\mathbb{R}^d  $.
    \item[$(L_\epsilon)$] A direct computation shows that the submanifold
    $$L_\epsilon = \left\{\tilde q = q + \epsilon \frac{\partial H}{\partial p}(q,\tilde p) \hbox{ and }
        \tilde p = p - \epsilon \frac{\partial H}{\partial q}(q,\tilde p) \right\}$$  
    is Lagrangian in $(\mathcal G, \Omega)$. Here $(q,p)$ and $(\tilde q, \tilde p)$ are cotangent coordinates on the first and second components of $\mathcal G $ respectively.
    \item[] For $\epsilon  $ small enough, it is also a bisection of $\mathcal G $, at least after restriction to a relatively compact open subset. To simplify the presentation, we will assume that it is a globally defined bisection. For $\epsilon=0 $, the bisection is the unit manifold $M= T^*\mathbb{R}^d $.

\item[$(\phi_\epsilon)$] The bisections $L_\epsilon $ define a smooth family of symplectomorphisms of $(M=T^*\mathbb{R}^d,\omega) $, which are precisely, for $\epsilon = \Delta t $, the symplectic Euler Poisson integrator \eqref{eq:PoissonBasic}.
\item[$(h_t,\mathbb H_\epsilon)$] 
Under the simplifying assumption that $H$ is separable, i.e. it splits in the following form:
$$H: T^*\mathbb{R}^d \to \mathbb{R}: (q,p) \mapsto V(q) + K(p)$$
with $V$ and $K$ two smooth real-valued functions, we can compute explicitely: 
\begin{enumerate}
\item[] the variation functions i.e. \begin{equation} \label{eq:ht}
    h_t \colon (q,p) \mapsto K(p) + V\left(q + t \frac{\partial K}{\partial p}(p)\right)  
\end{equation}
\item[] and the modified Hamiltonian is the Magnus series of $h_t$ with first two terms
\begin{equation}
\label{eq:Hepsilon}
    \mathbf H_\epsilon  := \epsilon H + \frac{\epsilon^2}{2} \left\langle \frac{\partial V}{\partial q}, \frac{\partial K}{\partial p} \right\rangle + \smallO{\epsilon^2}\\
\end{equation}
\end{enumerate}
\end{enumerate}

Let us explain how we computed \eqref{eq:ht}. Under the assumption on $H,$ the Lagrangian submanifold corresponding to \eqref{eq:PoissonBasic} becomes 

\begin{equation}
    \left\{
    \begin{array}{ll}
        \tilde q = q + \epsilon \frac{\partial K}{\partial p}(p - \epsilon \frac{\partial V}{\partial q})\\
        \tilde p = p - \epsilon \frac{\partial V}{\partial q}(q)
    \end{array}
\right. .
\end{equation}
Using the relations 

\begin{equation}
\left\{\begin{array}{ll}
    q = \tilde q - \epsilon  \frac{\partial K}{\partial p}(\tilde p) \\
    p = \tilde p + \epsilon \frac{\partial V}{\partial q}(\tilde q - \epsilon \frac{\partial K}{\partial p}(\tilde p))
\end{array} \right.,    
\end{equation}
the differentiation of the $L$-path $\phi_\epsilon =
\begin{pmatrix}
\tilde{q}\\
\tilde{p}
\end{pmatrix}$ with respect to $\epsilon$ gives :

$$\partial_\epsilon  \begin{pmatrix}
\tilde{q}\\
\tilde{p}
\end{pmatrix} = 
\begin{pmatrix}
\frac{\partial K}{\partial p}(\tilde p) - \epsilon \textbf{H}\text{ess}(K)(\tilde p). \frac{\partial V}{\partial q}(\tilde q - \epsilon \frac{\partial K}{\partial p}(\tilde p)) \\
\frac{\partial V}{\partial q}(\tilde q - \epsilon \frac{\partial K}{\partial p}(\tilde p))
\end{pmatrix}
= \begin{pmatrix}
\partial_p h_\epsilon (\tilde q, \tilde p) \\
-\partial_q h_\epsilon  (\tilde q, \tilde p)
\end{pmatrix}$$
and is consequently Hamiltonian with respect to the time-dependent Hamiltonian $h.$
\end{example}

\begin{example}[Mid-point scheme]\label{ex:mid_point}

For this example, the Hamiltonian is the harmonic oscillator $H\colon (q,p) \mapsto \frac{1}{2}(\|q\|^2 + \|p\|^2)$, for which it is well-known that the mid-point scheme

\[
    \left\{
    \begin{array}{ll}
        q_{n+1} = q_n + \Delta t \frac{p_{n} + p_{n+1}}{2}\\
        p_{n+1} = p_n - \Delta t \frac{q_{n} + q_{n+1}}{2} \\
    \end{array}
\right.,
\]

is symplectic.
\begin{enumerate}
    \item[$(M,\omega)$] The Poisson manifold $M= T^*\mathbb{R}^d$ is the same as the last example.
    \item[$(\mathcal G,\Omega)$] As a consequence, the symplectic groupoid $T^*\mathbb{R}^d \times T^*\mathbb{R}^d$ doesn't differ as well.
    \item[$(L_\epsilon)$] For any $\epsilon \in \mathbb{R},$ the submanifold 
    $$L_\epsilon = \left\{\tilde q = q + \epsilon \frac{p + \tilde{p}}{2} \hbox{ and }
        \tilde p = p - \epsilon \frac{q + \tilde{q}}{2} \right\}$$  
    is a Lagrangian bisection in $(\mathcal G, \Omega).$
    \item[$(h_t,\mathbb H_\epsilon)$] The above one-step forward map $\phi_t \colon (q,p) \mapsto (\tilde{q}, \tilde{p})$ induces a vector field $X_t = \partial_t \phi_t \circ \phi_t^{-1}$. One verifies that $X_t$ is colinear to the Hamiltonian vector field of $H,$ and so are $H$ and the time-dependent Hamiltonian:

$$ h_t(q,p) = \frac{(1-\frac{t^2}{4})^2 + t^2}{(1+ \frac{t^2}{4})^3}H(q,p).$$ The modified Hamiltonian is simply 
$$ \mathbb H_\epsilon  = \int_0^\epsilon \frac{(1-\frac{t^2}{4})^2 + t^2}{(1+ \frac{t^2}{4})^3} \text{d}t \times H$$
and the discretisation preserves $H$.
\end{enumerate}

\end{example}

\begin{example}[Linear Hamiltonian on the dual of a Lie algebra]
Let $G$ a Lie group, $\mathfrak{g}$ its Lie algebra and consider its symplectic groupoid $G \times \mathfrak{g}^* \rightrightarrows \mathfrak{g}^*.$ As the coadjoint action of $G$ on $\mathfrak{g}^*$ preserves the Lie bracket, a Poisson scheme discretising the flow of a linear Hamiltonian $f \in \mathfrak{g}$ is given by:

$$x_{n+1} = \text{Ad}_{\exp(\epsilon f)}^* x_n$$
and corresponds to the Lagrangian bisections $\{ (\exp (\epsilon f),x), x \in \mathfrak{g}^* \} \subset G \times \mathfrak{g}^*.$
\end{example}

\begin{example}[Kahan discretization of one Lotka-Volterra system]
\label{ex:Kahan}
For the quadratic Poisson bracket on $\mathbb{R}^d$ given by:
\begin{equation}
\label{eq:quadratic} \{x_i,x_j\} = x_i x_j \text{ if } 1 \leq i < j \leq d 
\end{equation}
and the linear Hamiltonian $H(x) = \sum_{i = 1}^d x_i,$ a Poisson scheme is given in \cite{vanhaecke2016} by the Kahan discretisation 
\begin{equation}\label{Kahan}
    x_i^{(n+1)}-x_i^{(n)} =  \Delta t \, x_i^{(n)} \Big(\sum_{j>i} x_j^{(n+1)} - \sum_{j<i} x_j^{(n+1)} \Big) +  \Delta t \, x_i^{(n+1)} \Big(\sum_{j>i} x_j^{(n)} - \sum_{j<i} x_j^{(n)}\Big)
\end{equation}
where $n$ is the iteration indice of the scheme and $x = (x_i)$ are coordinates on $\mathbb{R}^d$.

Let us interpret this discretization in terms of Theorem \ref{theorem_Poisson_integrators}:

\begin{enumerate}
    \item[$(M,\pi)$] is $M=\mathbb R^d $ with the Poisson structure \eqref{eq:quadratic}.
    \item[$(\mathcal G,\Omega)$] Its symplectic groupoid is 
    
    \noindent$\mathcal G = T^*\mathbb{R}^d,$\\
    $\Omega = \sum_{i} \text{d}x_i \wedge \text{d}p_i + 
        \sum_{i,j}(\delta_{i < j} - \delta_{i > j}) x_i p_j \text{d}x_i \wedge \text{d}p_j + \sum_{j < i} p_i p_j \text{d}x_i \wedge \text{d}x_j + \sum_{j < i} x_i x_j \text{d}p_i \wedge \text{d}p_j $\\
    $\alpha \colon (x,p) \mapsto x,$ \\
    $\beta \colon (x,p) \mapsto \left(e^{\sum_{i} (\delta_{i < j} - \delta_{i > j}) x_i p_i } x_j \right)_{1 \leq j \leq n}$\\
    
    \noindent with $(x,p)$ cotangent coordinates on $T^*\mathbb{R}^d.$

    \item[$(h_t)$] The variation function is given by:
$$h_t (x) = H(x) \times \frac{\partial f}{\partial t}(t,H(x)),$$
where
$$\begin{array}{ccccc}
f &: & (\mathbb{R},0)\times \mathbb{R} & \to & \mathbb{R} \\
 & & (t,u) & \mapsto & \frac{e^{u t} - 1}{u (e^{u t } + 1)} \\
\end{array}.$$ 
    \item[$(L_\epsilon)$] The family of Lagrangian submanifolds $L_\epsilon $ are given by $L_\epsilon = \Phi^\epsilon_{(\alpha^*h_t)_{t \in \mathbb{R}}}(\mathbb{R}^d).$
    
    $L_0$ is the unit manifold.
    \item[$(\phi_\epsilon)$] The induced Poisson diffeomorphism is precisely \eqref{Kahan} for $\epsilon =\Delta t$.
     \item[$(H_\epsilon )$] The modified Hamiltonian is simply $ \mathbf H_\epsilon (x):=  H(x) f(\epsilon, H(x)) $.    
    \end{enumerate}

Let us give some details on these points.

Let $\phi_\epsilon$ the map implicitly defined by \eqref{Kahan} and 

$\begin{array}{ccccc}
f &: & (\mathbb{R},0)\times \mathbb{R} & \to & \mathbb{R} \\
 & & (t,u) & \mapsto & \frac{e^{u t} - 1}{u (e^{u t } + 1)} \\
\end{array},$ then following proposition 3.1 of \cite{vanhaecke2016},
$$\phi_\epsilon(x) = \Phi_H^{f(t,x)}(x).$$
In this case, as 

$$ \Phi_H^{f(t,H(x))}(x) = \Phi_{h}^t(x)$$
where

$$h_t (x) = H(x) \times \frac{\partial f}{\partial t}(t,H(x)),$$
the Lagrangian bisection associated to \eqref{Kahan} are the image of $\mathbb{R}^d$ by the flows of the right-invariant vector fields associated to $\text{d}h$ in the symplectic groupoid of $(\mathbb{R}^d,\{.,.\}).$
\end{example}

\begin{example}[Splitting methods through Lagrangian bisections]
This example is inspired by \cite{Koseleff93}. Let $\groupoid$ be a symplectic groupoid integrating a Poisson manifold $M$ and $H = H_1 + H_2$ be a splitted Hamiltonian such that for each $H_i,$ one knows a smooth family of Lagrangian bisections $(L^i_t)_t$ inducing a Poisson integrator $(\phi^i_\epsilon)_\epsilon$ for $H_i$ at order $1$, with variation functions $(h_t^i)_t$.

\begin{enumerate}
    \item[$(L_\epsilon)$]  The composition of Lagrangian bisections in $\mathcal{G}$ is a Lagrangian bisection again, so that
    $$ t\mapsto  L^2_t \circ L^1_t$$
    is a family of Lagrangian bisections. It is easily checked to induce a Poisson integrator for $H$ at order $1$. 
    
    \item[($\phi_\epsilon$)] The induced Poisson diffeomorphism is the composition $\phi_\epsilon = \phi^1_\epsilon \phi^2_\epsilon.$
    
    \item[$(h_t)$] The variation function is $$h_t = h^1_t +  h^2_t  (\phi^1_{t})^{-1} $$ and equals $H$ at order $1.$
\end{enumerate}

For Poisson integrator at order $k$, the situation is more complicated. As shown by \cite{Koseleff93} in the symplectic context, we then have to compose several times the bisections, and use the following consequence of the Baker-Campbell-Hausdorff formula: assume we are given $\phi^1$ and $\phi^2$ two Poisson integrators for $H_1$ and $H_2$ at order 1, then there exists $n \in \mathbb{N}$ and $(c_l^j)_{\substack{j = 1,2 \\ 1 \leq l \leq n}}$ such that $L_\epsilon^{(k)} = \Pi_{l=n}^1 L^2_{c_l^2 \epsilon }.L^1_{c_l^1 \epsilon }$ induces a Poisson integrator for $H$ at order $k.$ 

\begin{enumerate}
    \item[$(L_\epsilon)$] Lagrangian bisections are $L_\epsilon^{(k)} = \Pi_{l=n}^1 L^2_{c_l^2 \epsilon }.L^1_{c_l^1 \epsilon }.$ 
    
    \item[($\phi_\epsilon$)] The induced Poisson diffeomorphism is the composition $\phi_\epsilon = \Pi_{l=1}^n \phi^1_{c_l^1 \epsilon } . \phi^2_{c_l^2 \epsilon }.$
    
    \item[$(h_t)$,$(H_\epsilon )$] The variation function is $$h_t = c_1^1 h^1_t + c^2_1 h^2_t  (\phi^1_{{c_1^1}t})^{-1} + c^1_2 h^1_t  (\phi^1_{{c_1^1}t}  \phi^2_{{c_1^2}t})^{-1}  + \ldots$$ and equals $H$ at order $k-1.$ The modified Hamiltonian is the Magnus series $\MagnusOne{h}$ of $h.$
\end{enumerate}

\end{example}

\section{Hamilton-Jacobi equation on the local symplectic groupoid}\label{sec:Ham_Jac}

When the symplectic groupoid is known, i.e. a symplectomorphism with $T^*M$ is given, constructive Poisson integrator of arbitrary order for an arbitrary Hamiltonian can be given. This will turn results of section \ref{sec:Poisson_integrators} into constructive ones.

\subsection{Geometry of Lagrangian bisections in the cotangent bundle of the base}\label{sec:Lag_bis_cotangent}

Let us recall a classical result of Poisson geometry:

\begin{theorem}
\cite{Weinstein1987}-\cite{Zung2005}-\cite{Crainic2011} There exists a neighborhood of $T^*M$ that carries a structure of local symplectic groupoid $\mathcal{G}$ on the base $M.$ Its symplectic form is the canonical one and its unit map is the zero section.
\end{theorem}
This theorem is, in its general form, an existence theorem. However, in many cases, the source and target maps of the groupoid structure on $(T^*M, \omega_{can}) $ can be made explicit.

We call a \emph{bi-realisation} of a Poisson manifold $(M,\pi)$ a triple $(\mathcal U, \alpha, \beta )$ made of a neighborhood of the zero section $\mathcal U \subset T^*M$ symplectomorphic to a local symplectic groupoid integrating $(M,\pi)$ such that the zero section corresponds to $M$ by this symplectomorphism and whose source and target are $\alpha $ and $\beta.$

\begin{remark}
Notice that, as explained in section \ref{sec:symp_grpid}, for any birealisation $(M,\alpha, \beta) $, the source $\alpha \colon \mathcal U \to M$ is a Poisson submersion and the target $\beta  \colon \mathcal U \to M$ an anti-Poisson submersion. Also, we do not specify the groupoid product.
\end{remark}

\noindent
Let us illustrate the notion of birealisation in some cases of interest. We make use of the so-called \emph{Poisson spray} of \cite{Crainic2011} and Moser's trick in a neighborhood of $M$ to compute bi-realisations of examples \ref{ex:groupoid_triv} and \ref{ex:Lotka-Volterra}.

\begin{example}\label{ex:groupoid_triv}
For the Poisson structure $\partial_p \wedge \partial_q$ of $T^*\mathbb{R}^n$ with coordinates $(q,p),$ denoting $(q,p,\xi_q,\xi_p)$ the induced coordinates on $T^*T^*\mathbb{R}^n,$ the choice of the Poisson spray $\xi_p \partial_q - \xi_q \partial_p$ gives the following birealisation:
\begin{equation}
    \left\{
    \begin{array}{ll}
        \alpha: (q,p,\xi_q,\xi_p) \mapsto (q - \frac{1}{2} \xi_p, p + \frac{1}{2} \xi_q)\\
        \beta: (q,p,\xi_q,\xi_p) \mapsto (q + \frac{1}{2}\xi_p, p - \frac{1}{2}\xi_q)
    \end{array}
    \right.
\end{equation}

\end{example}

\begin{example}
When $(M, \omega_M)$ is symplectic, there is no ``natural'' (i.e. preferred) way to send symplectically a neighborhood of the diagonal of the pair groupoid $(M \times M, p_1^* \omega_M -  p_2^* \omega_M)$ on a neighborhood of $M$ in $T^*M$. More precisely, there are as many ways as choices of Lagrangian bundles such that fibers are transverse to the diagonal in $M \times M$.
In fact, birealisations are in one-to-one correspondence with symplectomorphisms between a neighborhood of the zero section in $T^*M $ and a neighborhood of the diagonal in $M$. However, they may not be computable explicitly in general.
\end{example}

\begin{example}\label{groupoid_lin}
Let $G$ be a Lie group with Lie algebra $\mathfrak g $. 
Consider $\varphi$ a diffeomorphism from an open subset  $U \subset G$ to an open subset of $\mathfrak U \subset \mathfrak g$ mapping $1_G $ to $0$.

Since $\varphi $ is a diffeomorphism, $T \varphi \colon TU \to T\mathfrak U$ is an invertible vector bundle morphism, and so is $ T^* \varphi \colon T^*\mathfrak U \to T^* U$.
It is moreover a symplectomorphism, when $T^* \mathfrak U$ and $T^* U $ are equipped with their respective canonical structures. Since the source and target of $T^* G \simeq \mathfrak g^* \times G $ are given by $\alpha \colon (\xi,g) \mapsto \xi$  and $\beta \colon (\xi,g) \mapsto Ad_g^* \xi$, it suffices to transport those through $T^* \varphi$ to get a birealisation.

Let us be more explicit: with the cotangent lift
\begin{equation*}
    T^* \varphi:\underset{\xi_x}{T^*\mathfrak{g}}\underset{\mapsto}{\to}\underset{^{t}(d_{\varphi^{-1} x} \varphi).\xi_x}{T^*G}
\end{equation*}
and the natural isomorphism $T^*\mathfrak{g} \simeq T^*\mathfrak{g}^*,$ the symplectic groupoid of the dual of a Lie algebra $T^*G \rightrightarrows \mathfrak{g}^*$ becomes indeed $\mathfrak{g} \times \mathfrak{g}^*$ near $\mathfrak{g}^*$ with source and target:

\begin{equation}
    \left\{
    \begin{array}{ll}
        \alpha: (\mathfrak{g},0) \times \mathfrak{g}^* \to \mathfrak{g}^*: (\eta,\xi) \mapsto \left (L_{\varphi^{-1}(\eta)} {}^* T_{\varphi^{-1}(\eta)}^*\varphi \right).\xi\\
        \beta: (\mathfrak{g},0) \times \mathfrak{g}^* \to \mathfrak{g}^*: (\eta,\xi) \mapsto \left(R_{\varphi^{-1}(\eta)} {}^* T_{{\varphi^{-1}(\eta)}}^*\varphi \right).\xi
    \end{array}
    \right.
\end{equation}

The most natural diffeomorphism $\varphi $ is of course the logarithm map $\log: G \to \mathfrak g$. There are however other ones, like, e.g.:
\begin{enumerate}
    \item for $\mathfrak g $ the Lie subalgebra of $n \times n$ nilpotent matrices, the map $\varphi \colon x \to {\mathrm{id}}+x  $
    \item for $\mathfrak g $ the Lie subalgebra of  skew-symmetric $n \times n$  matrices, the map $x \mapsto \frac{{\mathrm{id}}+x/2}{{\mathrm{id}}-x/2} $ is also a diffeomorphism in a neighborhood of $0$.
\end{enumerate}
\end{example}

\begin{example}\label{ex:Lotka-Volterra}
The symplectic groupoid $\mathcal{G} \rightrightarrows \mathbb{R}^{n}$ of the real log-canonical Poisson bracket on $\mathbb{R}^n$, i.e.: 
\begin{equation}
\label{eq:quadratic2}
\{x_i, x_j \} = a_{ij} x_i x_j,
\end{equation}
with $(a_{ij})_{i,j}$ a skew-symmetric matrix is computed in \cite{Li2018} and is shown to be globally diffeomorphic to  $T^*\mathbb{R}^n$. The explicit structures given in \cite{Li2018} can be modified such that $\mathcal{G} =T^*\mathbb{R}^n$ is equipped with the canonical symplectic structure. The source and target maps defined in \cite{Li2018} then become, with $(x,p)$ cotangent coordinates on $T^*\mathbb{R}^n$:

\begin{equation}
    \left\{
    \begin{array}{ll}
        \alpha: (x,p) \mapsto \left( e^{-\frac{1}{2} \sum_i a_{ij} x_i p_i }. x_j \right)_{j=1, \dots, n}\\
        \beta: (x,p) \mapsto \left( e^{\frac{1}{2} \sum_i a_{ij} x_i p_i }. x_j \right)_{j=1, \dots, n}
    \end{array}
    \right.
\end{equation}
The triple $(T^*\mathbb{R}^n,\alpha,\beta)$ is a bi-realisation of the Poisson structure \eqref{eq:quadratic2}.
\end{example}

\subsection{Lagrangian bisections and Hamilton-Jacobi equation}
\label{sec:Ham_Jac_new}

We are now ready to use bi-realisations in order to look for Poisson integrators that approximate the flow of a Hamiltonian $H$, by considering them as graphs of closed 1-forms on $M$.

More precisely, assume we are given $(U,\alpha,\beta)$ a bi-realisation of a Poisson manifold $(M,\pi)$ and $H$ a Hamiltonian function. 
In the sequel, we will see from \eqref{HJ_Pois} that to the flow of $H$ corresponds a family $(L_t)_{t \in I} $ of Lagrangian bisections of the symplectic groupoid $(\mathcal{G}, \Omega) $, with $L_0=M $. Reducing $I$ if necessary, the bisections $(L_t)_{t \in I} $ become Lagrangian submanifolds in an open subset  $ U $ of $( T^* M, \omega_{can})$. Since $L_0$ is the zero section, $L_t$ is the graph of a closed $1$-form $ \zeta_t \in \mathcal C^\infty(M)$ depending smoothly on $t$.
This form is exact, thanks to the following proposition. The first and second points are consequences of proposition \ref{prop:11corresp} and example \ref{ex:Weinsteins_embedding} respectively.

\begin{proposition}
Let $\xi \in \Omega^1_0(U),$ $U \subset M$ an open subset and $I$ an open interval containing $0$ such that the flow of $\pi^{\#}(\xi)$ is defined for all $t \in I$ on $U.$
\begin{enumerate}
    \item There exists a unique smooth family of $1$-forms $(\zeta_t)_{t \in I}$ such that $\bar \zeta_t = \Phi^t_{\overrightarrow{\xi}}(M).$
    \item  For all $t,$ $\zeta_t$ is exact if and only if $\xi$ is also exact.
\end{enumerate}
\end{proposition}

\begin{corollary}
The closed $1$-forms are exact, $\zeta_t = d S_t$. 
with $(S_t)_t \in \mathcal{C}^\infty(M \times I)$ a solution of 
\begin{equation}\label{HJ_Pois}
    \left\{
    \begin{array}{ll}
        \partial_t S_t (m) &=  (\tau_{|\underline{dS_t}})^{-1} {}^* \alpha_{|\underline{dS_t}}^* H (m)  + \chi(t) \\
        S_0 &= 0
    \end{array}
    \right.
\end{equation}
where $ \chi(t) \in \mathcal C^\infty(I, \mathbb R)$ is any smooth function and $\tau $ is the cotangent projection.
\end{corollary}

We call \eqref{HJ_Pois} the
\emph{Hamilton-Jacobi equation} for a Poisson structure. 

\begin{remark}
Let us comment on the intial condition $S_0 = 0.$ In the context of this article, we are mainly interested with local embeddings of the symplectic groupoid $\mathcal{G}$ in some cotangent bundle $T^*V$ such that the unit space $M$ coincides with the base $V$. It may happen, though, that one considers embeddings where this property does not hold. For instance, the symplectic groupoid $T^*G$ of the dual of an integrable Lie algebroid $A = \text{Lie}(G)$ is naturally fibered on its groupoid $G$ but the fibration is transverse to the unit space. In those cases, it still makes sense to look for Lagrangian bisections as graphs of closed forms, but only if they are far from $M.$ There, one might relax the condition $S_0 = 0$ and build a family of Poisson automorphisms that are not perturbation of the identity map.
\end{remark}

\begin{theorem}\label{HJ_thm}
Assume we are given $(U,\alpha,\beta)$ a bi-realisation of a Poisson manifold $ (M,\pi)$ and $H$ a Hamiltonian function.
\begin{enumerate}
    \item The Hamilton-Jacobi equation \eqref{HJ_Pois} admits a solution $(S_t)_t$ in a neighborhood of $M \times \{0\} \subset M \times \mathbb R$.
    \item The family of Poisson automorphisms induced by the Lagrangian bisections $(dS_t)_t$ is the flow of $H$.
\end{enumerate}

\end{theorem}

\begin{proof}
The embedding of $\groupoid$ in $T^*M$ allows to express Lagrangian bisections near the base with graphs of closed 1-forms in a smooth way. That explains the first point.

Similar computation as \ref{ex:Weinsteins_embedding} gives the Hamiltonian induced by $(dS_t)_t,$ which admits the differential:
\begin{equation*}
    (\alpha_{|\underline{dS_t}}^{-1})^* \tau_{|\underline{dS_t}} {}^* d\partial_t S_t   = (\alpha_{|\underline{dS_t}}^{-1})^* d\overrightarrow{H} = dH.
\end{equation*}

\end{proof}

\begin{remark}
Let us relate the usual Hamilton-Jacobi equation described in Section \ref{sec:classic_Ham_Jac} with the equation we present in this section.

For the first one: $\Phi^t_{\overrightarrow{H}}(T^*Q) \subset T^*Q \times T^*Q$ is related by some graph of exact one-form $dS_t$ on $Q \times Q$ by $\Psi.$

For the second one: $\Phi^t_{\overrightarrow{H}}(M) \subset \mathcal{G}$ is related by some graph of exact one-form $dS_t$ on $M$ by the bi-realisation.

This equation is analogous to \eqref{eq:diff_class_Ham_Jac} in the sense of variations of Lagrangian bisections. Indeed, \eqref{eq:diff_class_Ham_Jac} measures Lagrangian perturbations of the diagonal in $T^*Q \times T^*Q$ by 1-forms on $Q \times Q$ through the canonical symplectomorphism (\ref{isomorphism2}) while the one of this section measures Lagrangian perturbations of $M$ in its local symplectic groupoid by 1-forms on $M$ through some bi-realisation.

\end{remark}

\subsection{Main result and numerical consequences}\label{sec:num_csq}

The computation of $(S_t)_t$ is not of interest from a numerical aspect because it is equivalent to integrate the Hamiltonian flow. Nevertheless, a natural consequence of theorem \ref{HJ_thm} is that the first terms of the expansion of $(S_t)_t$ with respect to $t$ induce an approximation of similar order of the flow of $H$:

\begin{theorem}\label{thm:num_csq}
Assume we are given $(U,\alpha,\beta)$ a bi-realisation of a Poisson manifold $ (M,\pi)$ and $H$ a Hamiltonian function. 
Define recursively a family $(S_i)_{i \in \mathbb{N}}$ of smooth functions on $M$ by
$S_0=0 $, $ S_1 = H$, $ S_2(m) = \frac{1}{2}\frac{d}{dt}|_{t=0} H(\alpha(td_mH) ) $, and 
\begin{equation}\label{eq:recursion}
    S_{i+1}(m) = \left. \frac{1}{(i+1) !}\frac{d^i}{dt^i}\right|_{t=0} H\left( \alpha\left(d_m S_t^{(i)} \right)\right)
\end{equation}
where we write $S_t^{(i)} =\sum_{j=1}^i t^j S_j.$

The family of Poisson automorphisms associated to the Lagrangian bisections $d \left( S_t^{(k)} \right) $ are Hamiltonian Poisson integrators of order $k$ for $H$ with variation functions :
\begin{equation}
    dh_t = (\tau_{|\underline{dS_t^{(k)}}} \circ \alpha_{|\underline{dS_t^{(k)}}}^{-1})^* d\partial_t S_t^{(k)}\\
\end{equation}
and the modified Hamiltonian verifies $\MagnusTwo{\epsilon}{h} = \epsilon H + \smallO{\epsilon^{k}}.$

\end{theorem}

\begin{remark}
The term of $S_t$ of order 1 in $t$ is necessarily $H.$
\end{remark}

Our general algorithm of a Poisson integrator of timestep $\Delta t$ for $H$ at order $k$, following remark \ref{rk:Ham_Poisson_integrator} and theorem \ref{thm:num_csq}, is given by the three steps:
\begin{enumerate}
    \item Use recursion \eqref{eq:recursion} to compute the $k$-th terms of $(S_t^{(k)})_t.$
    
    \item starting from $x \in M,$ solve 
    
    \begin{equation}\label{bisection}
        x = \alpha( d_{\overline{x}}S_{\Delta t}^{(k)}), \; \overline{x} \in M,
    \end{equation}
    \item and project
    
    \begin{equation}
        \tilde{x} = \beta( d_{\overline{x}}S_{\Delta t}^{(k)}).
    \end{equation}
\end{enumerate}

It is clear that for small $\Delta t,$ \eqref{bisection} always has a solution.

\subsection{Examples of Poisson integrators revisited}

\subsubsection{Poisson bracket of the canonical symplectic form of $T^*\mathbb{R}^n$}

The equation \eqref{HJ_Pois}, following choices of example \ref{ex:groupoid_triv}, becomes

\begin{equation}
    \partial_t S_t(q,p) = H\Big(q - \frac{1}{2}\partial_p S_t(q,p), p + \frac{1}{2}\partial_q S_t (q,p)\Big).
\end{equation}
and the corresponding numerical scheme is, starting from $(q,p) \in T^*\mathbb{R}^n:$

\begin{enumerate}[label=(\alph*)]\label{scheme_symp}
    \item solve 
    
    \begin{equation}
\left\{
    \begin{array}{ll}
        q = \overline{q} - \frac{1}{2}\partial_p S_{\Delta t}(\overline{q},\overline{p})  \\
        p = \overline{p} + \frac{1}{2}\partial_q S_{\Delta t} (\overline{q},\overline{p}) 
    \end{array}, \; (\bar q, \bar p) \in T^*\mathbb{R}^n
\right.
\end{equation}

    \item project

    \begin{equation}
\left\{
    \begin{array}{ll}
        \tilde{q} = \overline{q} + \frac{1}{2}\partial_p S_{\Delta t}(\overline{q},\overline{p})  \\
        \tilde{p} = \overline{p} - \frac{1}{2}\partial_q S_{\Delta t} (\overline{q},\overline{p}) 
    \end{array}
\right.
\end{equation}
\end{enumerate}

\begin{remark}
It is remarkable that the scheme \ref{scheme_symp} for the harmonic oscillator $H = \frac{x^2+y^2}{2}$ at order $1,$ i.e. for the lagrangian bisection given by the graph of $t \text{d}H,$ produces the mid-point scheme \ref{ex:mid_point}.
\end{remark}

\subsubsection{Linear Poisson bracket on the dual of a Lie algebra}

The equation \eqref{HJ_Pois}, following choices of example \ref{groupoid_lin}, becomes

\begin{equation}
    \partial_t S_t(x) = H(L_{\varphi^{-1} d_x S_t}^* T_{\varphi^{-1} d_x S_t}^*\varphi x)
\end{equation}

and the corresponding numerical scheme is given by the two steps: starting from $x \in \mathfrak{g}^*,$

\begin{enumerate}[label=(\alph*)]
    \item solve 
    
        \begin{equation*}
            \left(L_{\varphi^{-1} d_{\overline{x}} S_{\Delta t}}^* T_{\varphi^{-1} d_{\overline{x}} S_{\Delta t}}^*\varphi \right). \overline{x} = x, \; \bar x \in \mathfrak{g}^* 
        \end{equation*}
    
    \item project 
    
        \begin{equation*}
            \tilde{x} = \left(R_{\varphi^{-1} d_{\overline{x}} S_{\Delta t}}^* T_{\varphi^{-1} d_{\overline{x}} S_{\Delta t}}^*\varphi \right). \overline{x}
        \end{equation*}

\end{enumerate}

\subsubsection{Quadratic constant Poisson bracket}

Here, the Hamilton-Jacobi equation reads:

\begin{equation}
    \partial_t S_t (x) = H\Big( (e^{-\frac{1}{2} \sum_i a_{ij} x_i \partial_{x_i} S_t(x_i)} x_j)_j \Big)
\end{equation}

The term of first order in $t$ of $S_t$ is $H.$ The one of second order is :

\begin{equation}
    S_2(x) = -\frac{1}{2} \sum_{1 \leq i,j \leq n} a_{ij} \; x_i x_j \; \partial_{x_i}H(x)  \;   \partial_{x_j}H(x).
\end{equation}

In general, the functions $(S_t)_t$ correspond to the numerical scheme given by the two following steps: starting from $x \in \mathbb{R}^n$

\begin{enumerate}[label=(\alph*)]
    \item solve 
    
        \begin{equation*}
            \left(e^{-\frac{1}{2} \sum_i a_{ij} \overline{x_i} \partial_{x_i} S_{\Delta t}(\overline{x_i})} \overline{x_j}\right)_j = x_j \; \; \; \forall \; 1 \leq j \leq n
        \end{equation*}
    
    \item project 
    
        \begin{equation*}
            \tilde{x}_j = \left( e^{\frac{1}{2} \sum_i a_{ij} \overline{x_i} \partial_{x_i} S_{\Delta t}(\overline{x_i}) }. \overline{x_j} \right)_j  \; \; \; \forall \; 1 \leq j \leq n
        \end{equation*}
\end{enumerate}
to obtain a Poisson scheme of the quadratic Poisson bracket at any desired order.

\section*{Conclusion}

Let us sum up the message of this article. A bi-realisation of a Poisson manifold $M,$ i.e. a symplectomorphism between the local symplectic groupoid and a neighborhood of the base in $T^*M,$ allows to transform, through the analog of the Hamilton-Jacobi equation, a Hamiltonian $H \in \mathcal{C}^\infty(M)$ into a smooth family of functions $(S_t)_t$ on $M$ with $S_0=0$. Then, the recursively computed truncation $S^{(k)}$ of order $k$ of $S$ gives a Poisson integrator $\phi_{\Delta t}$ of order $k$ for $H,$ using the induced Lagrangian bisections $\underline{(\text{d}S_t)_t}$ and the source and targets: $\phi_{\Delta t} = \beta \circ (\alpha_{\underline{dS^{(k)}_{\Delta t}}})^{-1}$.  These integrators have strong geometric properties: not only their iterations stay on the symplectic leaf of the initial point (even a singular one), but they also follow the exact flow of a Hamiltonian on the manifold, which coincides with $H$ up to order $k-1$.

Hence the groupoid formalism developped in section \ref{sec:Poisson_integrators} proved to be useful for the construction of integrators. As one could expect, most existing Poisson integrators were already of that form, although not understood as such. Moreover, the Magnus formula introduced in section \ref{sec:Magnus} gives a new constructive way to compute the modified Hamiltonian of a Hamiltonian Poisson scheme and a new point of view on backward analysis in the context of geometric integrators for symplectic and Poisson geometry.

As mentioned in the introduction, one expects those integrators to be of particular interest in mechanics, where it matters to preserve properties of the dynamics when discretizing trajectories. In order to illustrate the link between their geometric properties and their long-term stability, we implement and benchmark Poisson schemes of section \ref{sec:num_csq} (\cite{Oscar2022}), to study them from a numerical aspect in comparison with other classical and geometric methods available to the community. There we also explain the ``minimal working knowledge'' in geometry to apply those to problems from mechanics.\\

\textbf{Acknowledgments.} I am deeply grateful to Camille Laurent-Gengoux and Vladimir Salnikov for constant attention while writing this article and acknowledge Chenchang Zhu, Aziz Hamdouni and Pol Vanhaecke for inspiring discussions at various stages of this work. 
This work has been supported by the CNRS 80Prime project ``GraNum'' and partially by the PHC Procope ``GraNum 2.0''.

\end{document}